\documentclass{daj}

\dajAUTHORdetails{%
  title = {Excluding Affine Configurations over a Finite Field}, %% please capitalize all significant words
  author = {Dion Gijswijt},
  plaintextauthor = {Dion Gijswijt},
  keywords = {additive combinatorics, extremal combinatorics, slice rank, Tur\'an number for matroids, affine configurations},
}   %%% END \dajAUTHORdetails

\dajEDITORdetails{%
   year={2023},
   number={21},
   received={4 April 2022},   % received date: example: 7 January 2017
   published={18 December 2023},  % published date
   doi={10.19086/da.91186},       % XXX = number of paper, e.g. da006 for paper#6
}   %%% END \dajEDITORdetails

\usepackage{amsmath,amssymb,amsthm,thmtools}

\newcommand{\NN}{\mathbb{N}}
\newcommand{\ZZ}{\mathbb{Z}}
\newcommand{\RR}{\mathbb{R}}
\newcommand{\F}{\mathbb{F}}
\newcommand{\qbin}[2]{{\genfrac{[}{]}{0pt}{}{#1}{#2}}_q}
\newcommand{\one}{\mathbf{1}}
\newcommand{\transp}{\mathsf{T}}
\newcommand\restr[2]{#1\raisebox{-.5ex}{$|$}_{#2}}
\newcommand{\Annbal}{\ensuremath{\Ann_{\bal}}}

\DeclareMathOperator{\rank}{rank}
\DeclareMathOperator{\arank}{arank}

\DeclareMathOperator{\spn}{span}
\DeclareMathOperator{\Ann}{Ann}
\DeclareMathOperator{\bal}{bal}
\DeclareMathOperator*{\rowspace}{row space}

\declaretheorem[style=definition]{definition}
\declaretheorem[style=plain,sibling=definition]{theorem}
\declaretheorem[style=plain,sibling=definition]{lemma}
\declaretheorem[style=plain,sibling=definition]{proposition}
\declaretheorem[style=plain,sibling=definition]{corollary}
\declaretheorem[style=plain,sibling=definition]{open problem}
\declaretheorem[style=remark,sibling=definition]{remark}
\declaretheorem[style=plain,numbered=no,name=Claim]{claim}
\declaretheorem[style=plain,title={Theorem}]{AlphTheorem}

\declaretheorem[style=plain,title={Corollary},numberlike=AlphTheorem]{AlphCorollary}

\newenvironment{proofclaim}{\par\noindent\textsl{Proof of claim.}\space}{\hfill $\scriptstyle\blacksquare$\par\smallskip}

\newcommand{\mytag}{\ensuremath{(\star)}}

\begin{document}

\begin{frontmatter}[classification=text]
%% EDITOR: this will force the keywords to appear right after the Abstract.
%%   If the abstract is too long and would force the keywords off the
%%   front page, please comment out % [classification=text] above
%%   This way the keywords will be floated on the bottom of the first page
%%   even though the Abstract spills over to the next page.

%%% AUTHOR: Title goes here.  This line is optional.  You must use it
%%   if title has footnote attached or requires nontrivial typesetting,
%%   e.g., inclusion of linebreaks to force nice layout.
\title{Excluding Affine Configurations over a Finite Field}
%% please capitalize all significant words

%%% AUTHOR:
%%% List all authors. If you wish, place grant acknowledgements in \thanks.
%%% In brackets include a short tag for each author.
\author[dion]{Dion Gijswijt}

%%% AUTHOR: Abstract goes here
\begin{abstract}
Let $a_{i1}x_1+\cdots+a_{ik}x_k=0$, $i\in[m]$ be a balanced homogeneous system of  linear equations with coefficients $a_{ij}$ from a finite field $\mathbb{F}_q$. We say that a solution $x=(x_1,\ldots, x_k)$ with $x_1,\ldots, x_k\in \mathbb{F}_q^n$ is \emph{generic} if every homogeneous balanced linear equation satisfied by $x$ is a linear combination of the given equations.  

We show that if the given system is \emph{tame}, subsets $S\subseteq \F_q^n$ without generic solutions must have exponentially small density. Here, the system is called tame if for every implied system the number of equations is less than half the number of used variables. Using a subspace sampling argument this also gives a `supersaturation result': there is a constant $c$ such that for $\epsilon>0$ sufficiently small, every subset $S\subseteq \F_q^n$ of size at least $q^{(1-\epsilon) n}$ contains $\Omega(q^{(k-m-\epsilon c)n})$ solutions as $n\to\infty$. For $q<4$ the tameness condition can be left out.

Our main tool is a modification of the slice rank method to leverage the existence of many solutions in order to obtain high rank solutions.
\end{abstract}
\end{frontmatter}

%%% AUTHOR: body of paper starts here
\section{Introduction}
Let $\F_q$ be a finite field. Consider a balanced homogeneous system of linear equations 
\begin{equation}\tag*{\mytag}
\begin{alignedat}{5}
        a_{11}x_1+&\cdots&&+a_{1k}x_k&&=0\\
        &\ \ \vdots\\
        a_{m1}x_1+&\cdots&&+a_{mk}x_k&&=0      
\end{alignedat}
\end{equation}
with coefficient matrix $A=(a_{ij})\in \F_q^{m\times k}$ and variable vectors $x_1,\ldots, x_k\in \F_q^n$.  By \emph{balanced}, we mean that $a_{i1}+\cdots+a_{ik}=0$ for all $i\in [m]$. In other words, the equations give affine dependencies between the $k$ variable vectors. We will always assume that the equations are linearly independent, i.e., $\rank(A)=m$. Note that for any $v\in \F_q^n$ we obtain a trivial solution by setting $x_i=v$ for all $i\in[k]$. 

If $x=(x_1,\ldots, x_k)$ is a solution to \mytag{}, we say that a set $S\subseteq \F_q^n$ contains the solution $x$ if $x_1,\ldots, x_k\in S$. For $n\to\infty$, we are interested in asymptotic upper bounds on the maximum size of a subset $S\subseteq \F_q^n$ that contains no `non-degenerate' solutions to \mytag. In particular, we will be interested in the case that $S$ can be shown to have exponentially small density in $\F_q^n$. In this context there are at least three natural notions of non-degeneracy that have previously been studied. A solution $x=(x_1,\ldots, x_k)$ to \mytag{} is said to be
\begin{itemize}
    \item[(a)] \emph{non-trivial} if $x_i\neq x_j$ for some $i,j\in[k]$;
    \item[(b)] \emph{a \mytag-shape\footnote{This terminology was introduced in \cite{Mimura-Tokushige-shape}.}} if $x_1,\ldots, x_k$ are pairwise distinct;
    \item[(c)] \emph{generic}\footnote{This terminology was introduced in \cite{JosseDion}} if $x$ only satisfies affine relations $\mu_1x_1+\cdots+\mu_kx_k=0$ that are linear combinations of the equations in~\mytag. 
\end{itemize}
Clearly, every \mytag-shape is a non-trivial solution. Every generic solution is a \mytag-shape, except when an equation of the form $x_{j_1}-x_{j_2}=0$ is implied by \mytag{} (in which case no \mytag-shape exists in~$\F_q^n$.) 

The question of how large a set without non-degenerate solutions can be, may be asked more generally for abelian groups instead of $\mathbb{F}_q^n$. Especially for $\ZZ/n\ZZ$ (or intervals in $\ZZ$) this is a central problem in additive number theory which includes famous results like Szemeredi's theorem on arithmetic progressions \cite{szemeredi1975sets}. Here we will solely focus on the case that the variables take values in $\mathbb{F}_q^n$. It is worth mentioning that \mytag{} always has a generic solution in $\F_q^{k-m-1}$. Hence, we have $|S|=o(q^n)$ for any $S\subseteq \F_q^n$ without generic solutions to \mytag{} as a direct consequence of the multi-dimensional density Hales-Jewett theorem~\cite{furstenberg1991density}. 

The system \mytag{} is called\footnote{The notions `moderate' and `temperate' were introduced in \cite{Mimura-Tokushige-shape} and \cite{JosseDion}, respectively.} 
\begin{itemize}
    \item \emph{moderate} if there is a $\delta>0$ such that $|S|=O(q^{(1-\delta)n})$ for any set $S\subseteq \F_q^n$ that does not contain a \mytag{}-shape,
    \item \emph{temperate} if there is a $\delta>0$ such that $|S|=O(q^{(1-\delta)n})$ for any set $S\subseteq \F_q^n$ that does not contain a generic solution to \mytag{}.
\end{itemize}
It follows that a temperate system is also moderate, provided the system does not imply an equation of the form $x_{j_1}-x_{j_2}=0$. In this paper we will show that under certain natural restrictions on \mytag, the system \mytag{} is temperate. 

Below, we mention some recent work on finding such exponential bounds for each of the three notions of non-degeneracy. In all cases, the main tool is the slice rank method or a variation thereof.

\paragraph{Non-trivial solutions}
Perhaps the simplest and most studied system is the single equation $x_1-2x_2+x_3=0$ where the solutions are precisely the $3$-term arithmetic progressions. In this case, the notions `non-trivial', `\mytag-shape' and `generic' coincide (for $q$ odd\footnote{If $q$ is even, the problem is not interesting. Indeed, $x_1,x_2,x_1$ is a `nontrivial' 3-AP for any $x_1\neq x_2$ and, moreover, no 3-AP $x_1,x_2,x_3$ exists with $x_1$, $x_2$, and $x_3$ pairwise distinct.}). A set $S\subseteq \F_3^n$ with no nontrivial $3$-term progressions is known as a \emph{cap set}. Based on the ground-breaking result of Croot, Lev, and Pach \cite{Croot-Lev-Pach}, it was shown in \cite{Ellenberg-Gijswijt} that $|S|\leq 3\cdot 2.756^n$ for any cap set $S\subseteq \F_3^n$. More generally, the authors show that any set $S\subseteq \F_q^n$ without non-trivial solutions to a given balanced equation $a_1x_1+a_2x_2+a_3x_3=0$ has exponentially small density in $\F_q^n$. In~\cite{kovac2021popular} the proof is adapted to upper bound sets in $(\F_q^2)^n$ without right isosceles triangles. 

The method used in~\cite{Ellenberg-Gijswijt}, was recast by Tao into a more elegant, symmetric form introducing the concept of \emph{slice rank}~\cite{Tao-symmetric, Tao-slice-rank}. Using the slice rank method, it is straightforward to generalize the cap set theorem to general balanced systems (see for example Theorem~1.1 in \cite{Sauermann} and also Proposition~4.12 in \cite{blasiak2016cap}). To describe the upper bound, let \[
m_{q,n,d}=|\{\alpha\in \{0,1,\ldots, q-1\}^n: \alpha_1+\cdots+\alpha_n\leq d\}|
\]
denote the number of monomials in $n$ variables of degree at most $d\in \mathbb{R}_{\geq 0}$ in which each variable has exponent at most $q-1$. Using the fact that for any integer $d$ the number of monomials of degree $d$ is the coefficient of $t^d$ in $(1+t+\dots+t^{q-1})^n$, one obtains the following upper bound.
\begin{lemma}
Let $\delta\in (0,1)$. Then  
\[
m_{q,n,\delta(q-1)n}\leq c_{q,\delta}^n, 
\]
where $c_{q,\delta}=\inf_{0<t\leq 1} (1+t+\cdots+t^{q-1})t^{-\delta(q-1)}$. Moreover, $c_{q,\delta}<q$ for $\delta<\tfrac{1}{2}$.
\end{lemma}

The slice rank method then yields the following result. 
\begin{theorem}\label{thm:slicerank}
Let $S\subseteq \F_q^n$ have no nontrivial solution to \mytag{}. Then 
\[
|S|\leq k\cdot m_{q,n,(q-1)nm/k}\leq k\cdot c_{q,m/k}^n.
\]
If $k\geq 2m+1$, we have $c_{q,m/k}<q$.
\end{theorem}
The factor $k$ in the theorem can be removed by the `power trick' since $S_1\times S_2$ has no nontrivial solution when the same holds for each $S_i\subseteq \F_q^{n_i}$. The cap set theorem corresponds to the case $(m,k)=(1,3)$. The theorem implies that if $k\geq 2m+1$, every set $S\subseteq \F_q^n$ without non-trivial solutions to~\mytag{} has exponentially small density. When $k\leq 2m$, the slice rank method breaks down since for $d\geq (q-1)n/2$ we have $m_{q,n,d}\geq \tfrac{1}{2}q^n$, so the obtained upper bound is trivial.

For $k\geq 4$ and any prime $p\geq k$, it is unknown whether every $S\subseteq \F_p^n$ without non-trivial $k$-term progressions, i.e., without non-trivial solutions to the system
\begin{equation*}
\begin{alignedat}{1}
x_1-2x_2+x_3&=0\\
\vdots\qquad\\
x_{k-2}-2x_{k-1}+x_k&=0
\end{alignedat}
\end{equation*}
has exponentially small density for $n\to\infty$.
 
\paragraph{\mytag-shapes}
In the case that \mytag{} consists of the single equation $x_1+\cdots+x_p=0$ with coefficients in $\F_p$, where $p\geq 3$ is a prime number, the maximum size of a set $S\subseteq \F_p^n$ without \mytag{}-shape is closely related to the Erd\H{o}s-Ginzburg-Ziv constant $\mathfrak{s}(\F_p^n)$ for the abelian group $\F_p^n$. The slice rank method does not (for $p\geq 5$) yield solutions with all $x_i$ distinct. However, using a variant called \emph{partition rank}, Naslund~\cite{naslund2020exponential} showed that $S\subseteq \F_p^n$ has exponentially small density if $S$ does not contain a solution to $x_1+\cdots+x_p=0$ with the $x_i$ pairwise distinct. Combining the slice rank method with other techniques, the exponent in the upper bound was further improved in \cite{fox2018erdHos, sauermann2021size}.  

Mimura and Tokushige were the first to explicitly study moderate systems, in part motivated by analogous results over the integers by Ruzsa \cite{ruzsa1993solving}. In \cite{Mimura-Tokushige-star, Mimura-Tokushige-shape, Mimura-Tokushige-II, mimura2021solving} they showed that several specific systems and families of systems are moderate. Perhaps the simplest example is the case where \mytag{} consists of a single balanced equation in $k\geq 3$ variables. A more general theorem containing these specific examples was proved in~\cite{JosseDion}. There, it is shown that if the coefficient matrix $A$ has many pairs of linearly dependent columns, the system will be moderate. On the opposite end, it was shown in \cite{Sauermann} that for $k\geq 3m$, systems in which all $m\times m$ minors of $A$ are nonzero, are moderate.

\paragraph{Generic solutions}  It is implicit in~\cite{Sauermann} that if \mytag{} consists of a single equation, the system is temperate. A more general result is shown in~\cite{JosseDion}. There, two columns of $A$ are said to be equivalent if they are scalar multiples of one another. It is shown that \mytag{} is temperate if either (i) for each equivalence class the columns sum to zero, or (ii) the number of classes is equal to $m+1$. If \mytag{} consists of a single equation, then there is only one equivalence class and we are in case (i).
  
In \cite{geelen2015analogue}, Geelen and Nelson proved a result in extremal matroid theory that is an analogue of the Erd\H{o}s-Stone theorem for graphs. The theorem states that for a fixed set $N$ of points in the projective space $\mathrm{PG}(m-1,q)$, the maximum size $\mathrm{ex}_q(N,n)$ of a subset $A\subseteq \mathrm{PG}(n-1,q)$ not containing a copy of $N$ satisfies $\frac{\mathrm{ex}_q(N,n)}{|\mathrm{PG}(n-1,q)|}\to 1-q^{1-\chi(N)}$ as $n\to\infty$, where $\chi(N)$ is the critical number of $N$. 

In the case that $N$ is affine ($\chi(N)=1$), this only gives an upper bound $o(q^n)$ on the Tur\'an number $\mathrm{ex}_q(N,n)$. However, for the binary case it was shown in \cite{bonin2000size} that in fact $\mathrm{ex}_2(N,n)=O(2^{\alpha n})$ for some $\alpha<1$ when $N$ is affine. This is equivalent to the statement that for $q=2$ the system \mytag{} is always temperate.

\subsection*{Main results}
We show that under certain conditions, the system \mytag{} is temperate. Our main result is the following. 
\begin{AlphTheorem}\label{thm:mainA}
The balanced system \mytag{} is temperate if the coefficient matrix $A$ is tame. 
\end{AlphTheorem}
Here, the matrix $A$ is called \emph{tame} if every system implied by \mytag{} that has $m'$ independent equations, uses at least $2m'+1$ variables. The precise definition of tameness can be found in Section~\ref{sec:tame}. We note that if \mytag{} implies the equation $x_1-x_2=0$, then without loss of generality, $A$ has the following form: 
\[
\arraycolsep=2pt
A = \left(\begin{array}{c|cc}
1&-1&0\cdots 0\\\hline
\begin{array}{c}0\\[-1ex]\vdots\\0\end{array}&\multicolumn{2}{c}{B}
\end{array}\right).
\]
Then \mytag{} is temperate if and only if the system with coefficient matrix $B$ is temperate. Hence, we may assume that \mytag{} does not imply an equation of the form $x_{j_1}-x_{j_2}=0$. 

The systems that were shown to be moderate or temperate in  \cite{JosseDion, Mimura-Tokushige-star, Mimura-Tokushige-shape,  Mimura-Tokushige-II, Sauermann} have tame coefficient matrices. Hence these results are implied by Theorem~\ref{thm:mainA} (albeit with possibly much worse exponents).

Over the fields $\F_2$ and $\F_3$, the tameness condition is not necessary.
\begin{AlphTheorem}\label{thm:mainB}
The balanced system \mytag{} is temperate if $q\in\{2,3\}$.
\end{AlphTheorem}

The number of solutions of affine rank $\leq r$ to \mytag{} in the space $\F_q^n$ is $O(q^{rn})$. In particular, the  total number of solutions is $O(q^{(k-m)n})$. As part of the proof of Theorem~\ref{thm:mainA} we show a supersaturation result for solutions of a fixed affine rank. Applied to Theorem~\ref{thm:mainA} and~\ref{thm:mainB}, we obtain the following corollary.  
\begin{AlphCorollary}\label{cor:mainD}
    Suppose that the coefficient matrix of \mytag{} is tame or that $q<4$. Then there exist $\delta>0$ and $C>0$ such that for all positive $\delta'<\delta$ the following holds: if $S\subseteq \F_q^n$ has size $|S|\geq q^{(1-\delta')n}$, then $S$ contains $\Omega(q^{(k-m-C\delta')n})$ solutions to \mytag{}. 
\end{AlphCorollary}

Theorems~\ref{thm:mainA} and~\ref{thm:mainB} can also be expressed in terms of excluding affine configurations.
\begin{AlphCorollary}\label{cor:mainC}
Let $z_1,\ldots, z_k\in \F_q^d$ be distinct. Suppose $q<4$ or that for every $i\in[k]$ the set $\{z_1,\ldots, z_k\}$ is the union of two affinely independent subsets containing $z_i$. Then every subset $S\subseteq \F_q^n$ containing no affine copy of $(z_1,\ldots, z_k)$ has exponentially small density in $\F_q^n$ (as $n\to \infty$).
\end{AlphCorollary}

\subsection*{Outline paper and proof}
In Section~\ref{sec:prelim}, we set up our notation and preliminaries. 

In Section~\ref{sec:tame}, we introduce \emph{tame} matrices. We give a good characterization for being tame using the matroid union theorem and show that any tame $m\times k$ matrix can be extended to a tame $m'\times (2m'+1)$ matrix. This allows us to restrict the proof of Theorem~\ref{thm:mainA} to the case $k=2m+1$. The main result in this section is Proposition~\ref{prop:disjointmaxrank}, which shows that for any solution $x$ of rank $r$ that is not generic ($r<(2m+1)-m$), there are two disjoint affine bases of $\{x_i:i\in [2m+1]\}$. This is a crucial part of the proof of Theorem~\ref{thm:mainA}, and the place where the tameness of $A$ is used.

In Section~\ref{sec:super} we show a supersaturation result for solutions of a given affine rank. If for some $\delta>0$ the condition $|S|\geq q^{(1-\delta)n}$ forces solutions of affine rank $\geq r$, then for $\delta'<\delta$ the condition $|S|\geq q^{(1-\delta')n}$ forces $S$ to have $\Omega(q^{r-c\delta'}n)$ solutions of affine rank $\geq r$ (for some constant $c$). This is used in the proof of Theorem~\ref{thm:mainA}. 

In Section~\ref{sec:thmA} we prove our main theorem: Theorem~\ref{thm:mainA}. It uses ideas from the slice rank method (in particular the Croot-Lev-Pach lemma), the supersaturation result and Proposition~\ref{prop:disjointmaxrank}. The main idea is to assume for contradiction that $r<(2m+1)-m$ is the maximum affine rank for which we can force solutions in $S$. We take the indicator function of the solution set $f:S^k\to \F_q$ and two suitable disjoint sets $I,J\subseteq [k]$ of size $r$. Summing $f$ against a suitable random $g:S^{[k]\setminus (I\cup J)}\to \F_q$, we obtain a function $M(f):S^I\times S^J\to \F_p$ that can be interpreted as an $|S|^r\times |S|^r$ matrix. Using the CLP lemma, the rank of $M(f)$ can be shown to be `small'. On the other hand, Proposition~\ref{prop:disjointmaxrank} implies that $M(f)$ is the sum of low rank matrices and an almost diagonal matrix that is `large' by the supersaturation result, resulting in a contradiction.

In Section~\ref{sec:thmB} we prove Theorem~\ref{thm:mainB}. The proof only uses the cap set theorem and an induction trick to lift it to a multi-dimensional variant.    

In Section~\ref{sec:aff} we briefly discuss the relation with affine configurations of points and show Corollary~\ref{cor:mainC}.

Finally, in Section~\ref{sec:conclusion} we conclude with Corollary~\ref{cor:mainD} and an open problem.

\section{Preliminaries and notation}\label{sec:prelim}
We denote the set of nonnegative integers by $\NN$. For any positive integer~$n$, we denote $[n]=\{1,\ldots, n\}$. Throughout the paper, $q$ will denote a (fixed) prime power and $\F_q$ denotes the field of order $q$. By $\F$ we will denote an arbitrary field. For a linear space $\F^n$ we denote by $e_1,\ldots, e_n$ the standard basis vectors. We will use the following basic lemma from linear algebra.
\begin{lemma}\label{lem:support}
Let $S$ be a set and let $V\subseteq \F^S$ be a linear subspace of dimension $d\in \NN$. Then there is a subset $T\subseteq S$ of size $|T|=d$ such that $f\mapsto \restr{f}{T}$ is an isomorphism between $V$ and $\F^T$. 
\end{lemma}

\subsection*{Asymptotic notation}
The asymptotic notations $O(\cdot)$, $o(\cdot)$, $\Omega(\cdot)$ used in this paper will always be with respect to the dimension~$n$ of the space $\F_q^n$. The notation $g(n)=\Omega(f(n))$ refers to the `Knuth big omega notation' common in computer science. In most cases, the function $g$ is implicit. For instance, by the statement `For all $S\subseteq \F_q^n$ with property $P$ we have $|S|=\Omega(g(n))$.' we mean that there exist a constant $C>0$ and an integer $n_0$ such that for all $n\geq n_0$ and every $S\subseteq \F_q^n$ with property $P$ we have $|S|\geq Cg(n)$.

\subsection*{Mulivariate polynomials}
For $\alpha\in \NN^n$ we denote by $x^\alpha=x_1^{\alpha_1}\cdots x_n^{\alpha_n}$ the monomial in variables $x_1,\ldots, x_n$ of \emph{degree} $|\alpha|:=\alpha_1+\cdots+\alpha_n$. More generally, if $\alpha=(\alpha_1,\ldots, \alpha_k)$, where $\alpha_1,\ldots, \alpha_k\in\NN^n$, then we write 
\[
|\alpha|:=|\alpha_1|+\cdots+|\alpha_k|=\sum_{i=1}^k\sum_{j=1}^n \alpha_{ij}\quad \text{and}\quad x^\alpha:=x_1^{\alpha_1}\cdots x_k^{\alpha_k}=\prod_{i=1}^{k}\prod_{j=1}^n x_{ij}^{\alpha_{ij}}.    
\]
Since $x^q=x$ for all $x\in \F_q$, we will restrict our polynomials to only have monomials in which all variables occur with exponent at most $q-1$. We write $\F_q[x_1,\ldots, x_n]'$ for the linear space spanned by the monomials $x^\alpha$, $\alpha\in \{0,1,\ldots, q-1\}^n$. For $d\in \RR$ we define  $M_{q,n,d}=\{\alpha\in \{0,1,\ldots, q-1\}^n:|\alpha|\leq d\}$ and $m_{q,n,d}=|M_{q,n,d}|$. 

The following key lemma is due to Croot, Lev and Pach (see \cite{Croot-Lev-Pach}).
\begin{lemma}[CLP lemma]
Let $f\in \F_q[x_1,\ldots, x_n,y_1,\ldots, y_n]'$ have degree at most $d$. Then the $q^n\times q^n$-matrix~$M$ given by $M_{ab}:=f(a,b)$ for all $a,b\in \F_q^n$, has rank at most $2m_{q,n,d/2}$. 
\end{lemma}

Since we will use similar arguments later, we give the proof here.
\begin{proof}
Write $f=\sum_{\alpha,\beta} c_{\alpha\beta}x^\alpha y^\beta$
where the sum is over all $\alpha,\beta\in\{0,\ldots,q-1\}^n$ that satisfy $|\alpha|+|\beta|\leq d$, and the $c_{\alpha\beta}\in \F_q$ are constants. Since each term has $|\alpha|\leq d/2$ or $|\beta|\leq d/2$, we can group the terms accordingly (if both $|\alpha|,|\beta|\leq d/2$ we can choose either group) and obtain
\[
f(x,y)=\sum_{|\alpha|\leq d/2} x^\alpha f_\alpha(y)+\sum_{|\beta|\leq d/2} y^\beta g_\beta(x)
\]
for certain $f_\alpha\in \F_q[y_1,\ldots, y_n]'$ and $g_\beta\in \F_q[x_1,\ldots, x_n]'$. Each of the terms $x^\alpha f_\alpha(y)$ and $y^\beta g_\beta(x)$ corresponds to a rank $1$ matrix (as it is the product of a function of $x$ and a function of $y$). Hence, $M$ has rank at most $2m_{q,n,d/2}$.
\end{proof}

\subsection*{Balanced linear systems}
We will consider systems of linear equations of the form 
\begin{equation}\tag*{\mytag}
\begin{alignedat}{5}
        a_{11}x_1+&\cdots&&+a_{1k}x_k&&=0\\
        &\ \ \vdots\\
        a_{m1}x_1+&\cdots&&+a_{mk}x_k&&=0      
\end{alignedat}
\end{equation}
where the coefficients $a_{ij}$ are elements of a field $\F$ and $x_1,\ldots,x_k\in \F^n$ are variable vectors. We will use the shorthand notation $x=(x_1,\ldots, x_k)\in \F^{n\times k}$, $A=(a_{ij})\in\F^{m\times k}$ and write the system of linear equations as\footnote{We can think of $x=(x_1,\ldots, x_k)$ as an $n\times k$ matrix. Then $Ax^\transp=0$, where we take the usual matrix product and $0$ denotes the $m\times k$ zero matrix, corresponds to system \mytag{}.} $Ax^\transp=0$. If $\sum_{j\in[k]} a_{ij}=0$ for all $i\in[m]$, we say that the matrix $A$ is \emph{row-balanced} and that the system $Ax^\transp=0$ is a \emph{balanced} system. 

\subsection*{Matroid terminology}
Let $A\in \F^{m\times k}$. For a subset $U\subseteq [k]$, we denote by $r_A(U)$ the rank of the submatrix of $A$ formed by the columns indexed by $U$. We have $r_A(U)=|U|$ if and only if the columns of $A$ indexed by $U$ are linearly independent. We say that $U$ is a \emph{basis} of $A$ if $|U|=r_A(U)=\rank(A)$. Note that any set $I\subseteq [k]$ with $r_A(I)=|I|$ is a subset of a basis. 

The pair $M=([k], \mathcal{I})$ where $\mathcal{I}=\{I\subseteq [k]: r_A(I)=|I|\}$ is a linear matroid with collection of independent sets $\mathcal{I}$ and rank function $r_A$. The following proposition follows directly by applying the matroid union theorem to two copies of $M$ (see for instance Corollary 42.1b in \cite{Schrijver}). It gives a min-max formula for the maximum size of the union of two sets of linearly independent columns of the matrix $A$.

\begin{proposition}\label{prop:matroidunion}
Let $A\in \F^{m\times k}$. Then 
\[
\max\{|I\cup J|: I,J\subseteq [k],\ r_A(I)=I,\ r_A(J)=J\}\ =\ \min\{k-|U|+2r_A(U): U\subseteq [k]\}. 
\]
\end{proposition}

\subsection*{Affine rank and generic solutions}
Given vectors $x_1,\ldots, x_k\in \F^n$, their \emph{affine rank} $\arank \{x_1,\ldots, x_k\}$ is the maximum number of affinely independent vectors in $\{x_1,\ldots, x_k\}$. In other words,
\[
\arank \{x_1,\ldots, x_k\}=\rank(X),\text{ where }X=\begin{pmatrix}1&\cdots&1\\x_1&\cdots&x_k\end{pmatrix}\in \F^{(n+1)\times k}.
\]
For $x=(x_1,\ldots, x_k)$ we will also denote $\arank(x)=\arank (x_1,\ldots, x_k)=\arank \{x_1,\ldots, x_k\}$. 

\begin{definition}
	For $(x_1,\ldots,x_k) \in (\F^n)^k$, let
	\[ \Annbal(x_1,\ldots, x_k) = \{(\mu_1,\ldots, \mu_k)\in \F_q^k : \mu_1x_1+\cdots+\mu_kx_k=0,\quad \mu_1+\cdots+\mu_k=0\}. \]
	So the elements of $\Annbal(x_1,\ldots, x_k)$ correspond to the affine relations between $x_1,\ldots, x_k$.
\end{definition}

\begin{definition}
If $x=(x_1,\ldots, x_k)$ is a solution to $Ax^\transp=0$, we say that $(x_1,\ldots, x_k)$ is a \emph{generic} solution if $\Annbal(x_1,\ldots, x_k)=\rowspace(A)$.
\end{definition}

\begin{lemma}
	\label{lem:affine-rank-nullity}
	Let $\F$ be a field and let $x_1,\ldots,x_k \in \F^n$. Then
	\[ \arank(x_1,\ldots,x_k) + \dim(\Annbal(x_1,\ldots,x_k)) = k. \]
\end{lemma}
\begin{proof}
	Let $A \in \F^{(n+1)\times k}$ be the matrix
	\[ A = \begin{pmatrix} 1 & \cdots & 1 \\ x_1 & \cdots & x_k\end{pmatrix}. \]
	For $I\subseteq [k]$ the vectors $x_i,\ i\in I$ are affinely independent if and only if the columns of $A$ indexed by $I$ are linearly independent. So $\rank(A) = \arank(x_1,\ldots,x_k)$. Since $\ker(A)=\Annbal(x_1,\ldots,x_k)$, the result follows from the rank-nullity theorem.
\end{proof}

\begin{lemma}\label{lem:generic}
    Let $A\in \F^{m\times k}$ be a row-balanced matrix of rank $m$. Let $x=(x_1,\ldots, x_k)$ be a solution to $Ax^\transp=0$. Then $\arank(x_1,\ldots,x_k) \leq k - m$. Equality holds if and only if $(x_1,\ldots,x_k)$ is a generic solution of $Ax^\transp=0$.
\end{lemma}
\begin{proof}
	Since $Ax^\transp=0$, the row space of $A$ is contained in $\Annbal(x_1,\ldots,x_k)$.
	Therefore we have $m = \rank(A) \leq \dim(\Annbal(x_1,\ldots,x_k))$, so it follows from Lemma~\ref{lem:affine-rank-nullity} that
	\[ \arank(x_1,\ldots,x_k) = k - \dim(\Annbal(x_1,\ldots,x_k)) \leq k - m. \]
	Clearly we have equality if and only if the row space of $A$ is equal to $\Annbal(x_1,\ldots,x_k)$, which is equivalent to $x$ being a generic solution. 
\end{proof}

\begin{lemma}\label{lem:lowdim}
Let $A\in \F^{m\times k}$ be a balanced matrix of rank $m$. Then $Ax^\transp=0$ has a generic solution in $\F^{k-m-1}$. 
\end{lemma}
\begin{proof}
By permuting the $k$ columns of $A$, and performing elementary row-operations, we may assume that $A$ is of the form $A=\begin{pmatrix}I&B\end{pmatrix}$. Let $z_{m+1},\ldots, z_k\in \F^{k-m-1}$ be affinely independent (e.g. $z_k=0$, $z_{m+i}=e_i$ for $i\in [k-m-1]$) and set $z_i=-\sum_{j\in [k-m]} B_{ij}z_j$ for $i\in [m]$. Then $z=(z_1,\ldots, z_k)$ is a solution to $Ax^\transp=0$ and $\arank(z)=k-m$. So by Lemma~\ref{lem:generic}, $z$ is a generic solution to $Ax^\transp=0$.  
\end{proof}

\begin{lemma}\label{lem:bases}
Let $A\in \F^{m\times k}$ be a row-balanced matrix of rank $m$. Let $x=(x_1,\ldots, x_k)$ be a solution to $Ax^\transp=0$ such that $\arank(x)=k-m$. Let $U\subseteq [k]$ have size $|U|=m$. Then 
\[
\text{$U$ is a basis of $A$} \iff \arank \{x_i:i\in [k]\setminus U\}=k-m. 
\]
\end{lemma}
\begin{proof}
Let $r=r_A(U)$ be the rank of the submatrix induced by the columns in $U$. By applying row operations to $A$ (and permuting columns in $U$) we may assume that this submatrix is equal to 
\[
\begin{pmatrix}I_r&B\\0&0\end{pmatrix}
\]
where $I_r$ is the $r\times r$ identity matrix. 

If $U$ is a basis of $A$, then $r=m$ and the system $Ax^\transp =0$ expresses for every $j \in U$ the vector $x_j$ as an affine combination of $\{x_i:i\in [k]\setminus U\}$. So $\arank \{x_i:i\in [k]\setminus U\}=\arank \{x_i:i\in [k]\}=k-m$.

If $U$ is not a basis of $A$, then $r<m$ and hence the equation $a_{m1}x_1+\cdots+a_{mk}x_k$ gives an affine relation between the vectors in $\{x_i:i\in[k]\setminus U\}$ as $a_{mj}=0$ for all $j\in U$. So $\arank \{x_j:j\in [k]\setminus U\}<k-m$.
\end{proof}

\section{Tame matrices}\label{sec:tame}
To show that a system $Ax^\transp=0$ is temperate, we will rely on the slice rank method not only for the system $Ax^\transp=0$ itself, but also for implied systems. For this reason, we want the number of variables in an implied system to be more than twice the number of equations. If this is the case, we say that $A$ is \emph{tame}. 

\begin{definition}
Let $A\in \F^{m\times k}$ be a matrix of rank $m$. We say that $A$ is \emph{tame} if, by elementary row-operations and column permutations, we cannot bring $A$ into the form  $\left(\begin{smallmatrix}A'&0\\ B & C \end{smallmatrix}\right)$ where $A'\in \F^{m'\times k'}$ and $k'\leq 2m'$. 
\end{definition}

\begin{lemma}\label{tamechar}
Let $A\in \F^{m\times k}$ be a matrix of rank $m$. Then, the following are equivalent:
\begin{itemize}
    \item[\textup{(i)}] $A$ is tame.
    \item[\textup{(ii)}] For all $U\subsetneq [k]$ we have $2r_A(U)\geq 2m+1-k+|U|$. 
    \item[\textup{(iii)}] For all $i\in [k]$ the set $[k]\setminus\{i\}$ contains two disjoint bases of $A$.
\end{itemize}
\end{lemma}

\begin{proof}
\item $\textup{(ii)}\Rightarrow \textup{(i)}$.\quad 
Assume (ii) holds, and suppose that $A$ is not tame. Using elementary row-operations and permuting columns, we can bring $A$ into the form $\tilde{A}=\left(\begin{smallmatrix}A'&0\\ B & C \end{smallmatrix}\right)$ where $A'\in \F^{m'\times k'}$ with $k'\leq 2m'$. Since $r_A$ is preserved under elementary row-operations, the matrix $\tilde{A}$ also satisfies (ii). However, taking $U=\{k'+1,\ldots, k\}$ we obtain a contradiction. Indeed, since $C$ has rank at most $m-m'$, we have $r_{\tilde{A}}(U)\leq m-m'$ and therefore 
\[
2r_{\tilde{A}}(U)\leq 2m-2m'\leq 2m-k'=2m-k+|U|.
\]

\item $\textup{(i)}\Rightarrow \textup{(ii)}$.\quad 
Suppose that \textup{(ii)} does not hold and let $U\subsetneq [k]$ be such that $2r_A(U)\leq 2m-k+|U|$. Set $k'=k-|U|$ and $m'=m-r_A(U)$. Then 
\[
k'=k-|U|\leq 2m-2r_A(U)=2m'.
\]
By permuting columns of $A$, we may assume that $U=\{k'+1,\ldots, k\}$. By performing elementary row-operations, we can bring $A$ into the form $\left(\begin{smallmatrix}A'&0\\B&C\end{smallmatrix}\right)$, where $C$ is a $|U|\times r_A(U)$ matrix. Then $A'$ is an $m'\times k$ matrix with $k'\leq 2m'$.

\item $\textup{(ii)}\Rightarrow \textup{(iii)}$.\quad 
Let $i\in [k]$. For every $U\subseteq [k]\setminus\{i\}$ we have $2r_A(U)\geq 2m+1-k+|U|$, or equivalently, \[(k-1)-|U|+2r_A(U)\geq 2m.\] By Proposition~\ref{prop:matroidunion} there exist $I,J\subseteq [k]\setminus\{i\}$ such that $|I\cup J|=2m$ and $r_A(I)=|I|$ and $r_A(J)=|J|$. Hence, $I$ and $J$ are two disjoint bases of $A$.

\item $\textup{(iii)}\Rightarrow \textup{(ii)}$.\quad
Let $U\subsetneq [k]$ and let $i\in[k]\setminus U$. Let $B,B'\subseteq [k]\setminus\{i\}$ be two disjoint bases. Then 
\[|U\cap (B\cup B')|\geq |U|-(k-1-|B\cup B'|)=|U|-k+1+2m.\]
It follows that 
\[2r_A(U)\geq 2\max\{|U\cap B|,|U\cap B'|\}\geq |U\cap B|+|U\cap B'|\geq |U|-k+1+2m.\] 
\end{proof}

\begin{remark}Note that Lemma~\ref{tamechar} gives a good characterization of tame matrices in the following sense: if a matrix is not tame, this can be certified by giving a set $U\subsetneq [k]$ that violates the inequality in (ii), and if a matrix is tame, this can be certified by giving two disjoint bases contained in $[k]\setminus \{i\}$ for every $i\in [k]$ as in (iii). In fact, since finding a maximum size set that is the union of two independent sets can be done in polynomial time by the matroid union algorithm, we can check in polynomial time if a given matrix is tame by checking (iii).   
\end{remark}

\begin{lemma}\label{extendone}
Let $A\in \F^{(m+1)\times(2m+1)}$ have rank $m+1$, and let $A'$ be obtained by deleting the last row of $A$. Suppose that $A'$ is tame. Then, there are bases $B_1,B_2$ of $A$ such that $B_1\cup B_2=[2m+1]$. 
\end{lemma}
\begin{proof}
Let $U\subseteq [2m+1]$. If $|U|=2m+1$, then $r_A(U)=m+1$. If $|U|=2\ell+1$ for some $\ell\in \{0,\ldots, m-1\}$, we have $r_A(U)\geq r_{A'}(U)\geq \ell+1$ by Lemma~\ref{tamechar} since $A'$ is tame.  If $|U|=2\ell+2$ and $u\in U$, we have $r_A(U)\geq r_A(U\setminus\{u\})\geq \ell+1$. So we see that
\[
2r_A(U)\geq |U|\quad\text{for all $U\subseteq [2m+1]$}. 
\]
It now follows by Proposition~\ref{prop:matroidunion} that there are $I_1,I_2\subseteq [2m+1]$ with $r_A(I_1)=|I_1|$, $r_A(I_2)=|I_2|$, and $|I_1\cup I_2|=2m+1$. Extending for $i=1,2$ the set $I_i$ to a basis $B_i$ of $A$ completes the proof.
\end{proof}

The following proposition will be of key importance in the proof of Theorem~\ref{thm:mainA}.
\begin{proposition}\label{prop:disjointmaxrank}
Let $A\in \F^{m\times (2m+1)}$ be a balanced, tame matrix of rank $m$. Let $x$ be a solution to $Ax^\transp=0$ with $\arank(x)=r\leq m$. Then there exist disjoint subsets $I_1,I_2\subseteq [2m+1]$ of size $r$, such that $\arank\{x_i:i\in I_1\}=\arank\{x_i:i\in I_2\}=r$.
\end{proposition}
\begin{proof}
Let $m'=2m+1-r\geq m+1$. We can extend $A$ with additional rows to obtain a balanced matrix $A'\in \F^{m'\times (2m+1)}$ of rank $m'$ such that $A'x^\transp=0$ (extend the rows of $A$ to a basis of $\Annbal(x)$.) Note that we add $m+1-r\geq 1$ rows in total.

Let $A''$ be the submatrix of $A'$ consisting of the first $m+1$ rows (so $A$ with one row added). By Lemma~\ref{extendone} there exist subsets $J_1,J_2\subseteq [2m+1]$ of size $m+1$ such that $J_1\cup J_2=[2m+1]$ and $r_{A''}(J_1)=r_{A''}(J_2)=m+1$. Clearly, we then also have: $r_{A'}(J_1)=r_{A'}(J_2)=m+1$. It follows that there are bases $B_1$ and $B_2$ of $A'$ such that $B_1\cup B_2=[2m+1]$. Since $|B_1|=|B_2|=2m+1-r$, it follows that $I_1:=[2m+1]\setminus B_1$ and $I_2:=[2m+1]\setminus B_2$ are disjoint sets of size $r$. By Lemma~\ref{lem:bases}, $\arank\{x_i:i\in I_1\}=r=\arank\{x_i:i\in I_2\}$. 
\end{proof}

\subsection*{Extending tame matrices}
In this subsection we show the following proposition.
\begin{proposition}\label{prop:extend}
Let $A\in \F^{m\times k}$ be tame and row-balanced with $k>2m+1$. Then $A$ can be extended to a tame, row-balanced matrix $A'\in \F^{m'\times (2m'+1)}$ such that if $x'=(x_1,\ldots, x_{k'})$ is a generic solution to $A'(x')^\transp=0$, then $x=(x_1,\ldots, x_k)$ is a generic solution to $Ax^\transp=0$. 
\end{proposition}
Before proving the proposition, we need a number of lemmas.

\begin{lemma}\label{Extend_Tame_Matrix}
Let $A\in \F^{m\times k}$ be tame. Suppose that $k\geq 2m+2$. Then there exist $i\in[k-1]$ such that for all $\alpha,\beta,\gamma\in \F\setminus\{0\}$ the matrix 
\[
A'=\begin{pmatrix}A&0\\\alpha e_i^\transp+\beta e_k^\transp&\gamma\end{pmatrix}\in \F^{(m+1)\times (k+1)}
\]
is tame. 
\end{lemma}
\begin{proof}
For $U\subseteq [k-1]$ define $f(U)=2r_A(U)-2m-1+k-|U|$. Since $A$ is tame, it follows by Lemma~\ref{tamechar} that $f(U)\geq 0$ for every $U\subseteq  [k-1]$. Let 
\[
\mathcal{T}=\{U\subseteq [k-1]: f(U)=0\}.
\]
Note that $[k-1]\in \mathcal{T}$ since $r_A([k-1])\leq m$. So $\mathcal{T}$ is nonempty.

\begin{claim} Let $U,V\in \mathcal{T}$. Then $U\cup V, U\cap V\in \mathcal{T}$.
\end{claim}
\begin{proofclaim}
By the submodularity of $r_A$ we have $f(U\cap V)+f(U\cup V)\leq f(U)+f(V)=0$. Since $f(U\cap V)\geq 0$ and $f(U\cup V)\geq 0$, we have $f(U\cap V)=f(U\cup V)=0$.
\end{proofclaim}

Let $I\subseteq [k-1]$ be the unique inclusionwise minimal subset of $\mathcal{T}$. Since $\emptyset\not\in \mathcal{T}$ as $k> 2m+1$, the set $I$ is nonempty. Let $i\in I$. So $i\in U$ for all $U\in \mathcal{T}$. We will show that for this choice of $i$ the matrix $A'$ is tame for all $\alpha, \beta,\gamma\in \F\setminus\{0\}$.

Let $U\subsetneq [k+1]$ be arbitrary. By Lemma~\ref{tamechar}, it suffices to show that 
\[2r_{A'}(U)\geq 2(m+1)+1-(k+1)+|U|.\] 
Suppose that for some $j\in \{i,k,k+1\}$ we have $j\in U$. If $k+1\in U$ we choose $j=k+1$. Note that $U\setminus\{j\}\subsetneq [k]$. Since $\alpha,\beta,\gamma\neq 0$, we have $r_{A'}(U)\geq 1+r_A(U\setminus\{j\})$. Since $A$ is tame, we obtain 
\begin{align*}
    2r_{A'}(U)&\geq 2+2r_A(U\setminus\{j\})\\
    &\geq 2+2m+1-k+|U\setminus\{j\}|\\
    &=2(m+1)+1-(k+1)+|U|
\end{align*}
as desired.  

Therefore, we may assume that $U\cap\{i,k,k+1\}=\emptyset$. So $U\subseteq [k-1]$ and $U\not\in \mathcal{T}$. It follows that $f(U)\geq 1$, and therefore
\begin{align*}
   2r_{A'}(U)&\geq 2r_A(U)\\
   &\geq 1+ (2m+1-k+|U|)\\
   &=2(m+1)+1-(k+1)+|U|. 
\end{align*}
\end{proof}

Note that if $\F\neq \F_2$, we can choose elements $\alpha,\beta,\gamma\in \F\setminus\{0\}$ such that $\alpha+\beta+\gamma=0$. Hence, if $A$ is row-balanced, we can take $A'$ to be row-balanced as well. For $\F=\F_2$ this does not work. In fact the row-balanced tame matrix $\begin{bmatrix}1&1&1&1\end{bmatrix}\in \F_2^{1\times 4}$ cannot be extended to a row-balanced tame matrix $A'\in \F_2^{2\times 5}$. For this reason, we need the following lemma.

\begin{lemma}\label{Extend_Tame_Matrixtwo}
Let $A\in \F_2^{m\times k}$ be tame. Suppose that $k\geq 2m+2$. Then there exists $i\in[k-1]$ such that the matrix $A''\in \F_2^{(m+3)\times (k+5)}$ given by
\[
A'':=\begin{pmatrix}A&0&0&0&0&0\\e_i^\transp+e_k^\transp &1&0&1&0&0\\e_i^\transp  &1&1&0&1&0\\e_k^\transp &1&1&0&0&1\end{pmatrix}.
\]
is tame. 
\end{lemma}
Note that if $A$ has rank $m$ and is row-balanced, then $A''$ has rank $m+3$ and is row balanced.
\begin{proof}
Let $i$ be as in the conclusion of Lemma~\ref{Extend_Tame_Matrix}. So the submatrix 
\[A'=\begin{pmatrix}A&0\\e_i^\transp +e_k^\transp &1\end{pmatrix}\]
of $A''$ induced by the first $k+1$ columns and $m+1$ rows is tame. We will use the shorthand notations $r'=r_{A'}$ and $r''=r_{A''}$. By Lemma~\ref{tamechar} we know that 
\[2r'(U)\geq 2(m+1)+1-(k+1)+|U|\quad\text{for all $U\subsetneq [k+1]$}.\]
Let $U\subsetneq [k+5]$ be arbitrary. To show that $A''$ is tame, it suffices to show that 
\[2r''(U)\geq 2(m+1)+1-(k+1)+|U|\] since $2(m+3)+1-(k+5)=2(m+1)+1-(k+1)$. 

By permuting the first $k-1$ columns, we may assume without loss of generality that $i=k-1$. Let $U_1=U\cap [k-2]$ and $U_2=U\setminus U_1=U\cap \{k-1,\ldots, k+5\}$. First consider the case that $|U_2|\leq 6$. Note that the submatrix of $A''$ induced by the last three rows and the last seven columns represent all $7$ vectors in $\F_2^3\setminus\{0\}$. Since any set of $t\leq 6$ vectors in $\F_2^3\setminus\{0\}$ spans a subspace of dimension at least $t/2$, we see that $r''(U)\geq r'(U_1)+|U_2|/2$. Hence,
\begin{align*}
2r''(U) &\geq 2r'(U_1)+|U_2|\\
&\geq 2(m+1)+1-(k+1)+|U_1|+|U_2|\\
&= 2(m+1)+1-(k+1)+|U|.
\end{align*}

Now consider the remaining case that $|U_2|=7$. Since $k+4$ and $k+5$ belong to $U$,  we have $r''(U)\geq 2+r'(U\cap [k+1])$. Since $U$ does not contain $[k+1]$ (as $U\not=[k+5])$, we have
\begin{align*}
2r''(U)&\geq 2r'(U\cap [k+1])+4\\
&\geq 2(m+1)+1-(k+1)+|U\cap [k+1]|+4\\
&=2(m+1)+1-(k+1)+|U|.
\end{align*}
\end{proof}

We are now ready to prove Proposition~\ref{prop:extend}.
\begin{proof}
It suffices to show that $A$ can be extended to a row-balanced, tame $k'\times m'$ matrix $A'$ such that the following two properties hold:
\begin{itemize}
    \item $k'-2m'=(k-2m)-1$,
    \item if $x'=(x_1,\ldots, x_{k'})$ satisfies $A'(x')^\transp=0$ and $\arank\{x'\}=k'-m'$, then $x=(x_1,\ldots, x_k)$ satisfies $Ax^\transp=0$ and $\arank(x)=k-m$.
\end{itemize}
Indeed, repeating this step $k-2m-1$ times results in a $(k-m-1)\times (2(k-m-1)+1)$ matrix with the required properties. Note that we used the characterization of generic solutions from Lemma~\ref{lem:generic}.

Let us first consider the case that $\F\neq \F_2$. Let $\alpha,\beta,\gamma\in \F\setminus\{0\}$ be such that $\alpha+\beta+\gamma=0$. Let $m'=m+1$ and $k'=k+1$. By Lemma~\ref{Extend_Tame_Matrix}, there exists $i\in [k-1]$ such that the matrix $A'\in \F^{m'\times k'}$ given by
\[A'=\begin{pmatrix}A&0\\\alpha e_i+\beta e_k&\gamma\end{pmatrix}\]
is tame. Since $\alpha+\beta+\gamma=0$, the matrix $A'$ is balanced. Now suppose $x'=(x_1,\ldots, x_{k+1})$ satisfies $A'(x')^\transp=0$ and $\arank(x')=(k+1)-(m+1)$. Let $x=(x_1,\ldots, x_k)$. Clearly, $Ax^\transp=0$. Since $\gamma\neq 0$, the vector $x_{k+1}$ is an affine combination of $x_1,\ldots, x_k$, so $\arank(x_1,\ldots, x_k)=\arank(x_1,\ldots, x_{k+1})=k-m$ as required.

Now let us consider the case that $\F=\F_2$. By Lemma~\ref{Extend_Tame_Matrixtwo}, we obtain the tame, row-balanced matrix $A'\in \F_2^{m'\times k'}$ with $m'=m+3$ and $k'=k+5$ given by
\[
A':=\begin{pmatrix}A&0&0&0&0&0\\e_i+e_k&1&0&1&0&0\\e_i&1&1&0&1&0\\e_k&1&1&0&0&1\end{pmatrix},
\]
where $i\in [k-1]$. Let $x'=(x_1,\ldots, x_{k'})$ satisfy $A'(x')^\transp=0$ and $\arank(x')=k'-m'$. Let $x=(x_1,\ldots, x_k)$. Clearly, we have $Ax^\transp=0$. Since $x_{k+3}$, $x_{k+4}$ and $x_{k+5}$ are affinely dependent on $x_1,\ldots, x_{k+2}$, we have $\arank(x')\leq \arank(x_1,\ldots, x_{k+2})\leq \arank(x)+2$. We obtain
\[
\arank(x)\geq \arank(x')-2=(k+5)-(m+3)-2=k-m.
\]
Note that by Lemma~\ref{lem:generic} we also have $\arank(x)\leq k-m$, so we have $\arank(x)=k-m$ as required.
\end{proof}

\section{supersaturation}\label{sec:super}
In this section we prove a supersaturation result (Proposition~\ref{supersatnew}) that is essential to the proof of our main theorem. The proof is based on a standard subspace sampling argument (see for instance Exercise~10.1.9 of \cite{tao2006additive} or \cite{fraser2021three}). We begin by stating some preliminaries. 

For $d,n\in \NN$ the \emph{Gaussian binomial coefficient} $\qbin{n}{d}$ is given by
\[
\qbin{n}{d}=\begin{cases}\prod_{i=0}^{d-1}\frac{1-q^{n-i}}{1-q^{d-i}}&\text{if }d\leq n\\0&\text{if }d>n\end{cases}.
\]
The following fact is well-known.
\begin{lemma}\label{lem:countsubspaces}
Let $k\leq d\leq n$ be nonnegative integers. Then the number of $d$-dimensional linear subspaces of $\F_q^n$ containing a given $k$-dimensional linear subspace is $\qbin{n-k}{d-k}$. 
\end{lemma}
\begin{proof}
Let $V\subseteq \F_q^n$ be a $k$-dimensional linear subspace and let $b_1,\ldots, b_k$ be a basis of $V$. The number of $d-k$-tuples $(b_{k+1},\ldots, b_d)$ such that $b_1,\ldots, b_d$ are linearly independent vectors in $\F_q^n$, is equal to 
$(q^n-q^k)(q^n-q^{k+1})\cdots(q^n-q^{d-1})$. Indeed, for $i=k+1,\ldots, d$ we can choose $b_i$ arbitrarily in $\F_q^n\setminus \mathrm{span}(b_1,\ldots, b_{i-1})$. Thus, we obtain a $d$-dimensional subspace $W:=\mathrm{span} (b_1,\ldots, b_d)$ containing~$V$. 

Each subspace $W\supseteq V$ of dimension $d$ is thus obtained $(q^d-q^k)\cdots(q^d-q^{d-1})$ times, since there are that many ways to extend $b_1,\ldots, b_k$ to a basis $b_1,\ldots, b_d$ of $W$. We conclude that the number of $d$-dimensional subspaces containing $V$ is equal to  
\[
\frac{(q^n-q^k)\cdots(q^n-q^{d-1})}{(q^d-q^k)\cdots(q^d-q^{d-1})}=\qbin{n-k}{d-k}.
\]
\end{proof}
Note that Lemma~\ref{lem:countsubspaces} implies that the number of $d$-dimensional affine subspaces of $\F_q^n$ that contain a given subset $S\subseteq \F_q^n$ of affine rank $r$, is equal to $\qbin{n-r+1}{d-r+1}$. Indeed, by shifting, we may assume that $0\in S$ and then $\dim\spn(S)=r-1$.

We will need the following bounds on the Gaussian binomial coefficients.
\begin{lemma}\label{lem:qbinbound}
Let $d\leq n$ be nonnegative integers. Then 
\[
q^{d(n-d)}\leq \qbin{n}{d}\leq 4q^{d(n-d)}.
\]
\end{lemma}
\begin{proof}
The lower bound follows directly from the fact that for all $i\in\{0,\ldots, d-1\}$ we have 
\[\frac{q^{n-i}-1}{q^{d-i}-1}\geq \frac{q^{n-i}}{q^{d-i}}=q^{n-d}.\]

The upper bound follows from the fact that for all $i\in \{0,\ldots, d-1\}$ we have
\[\frac{q^{n-i}-1}{q^{d-i}-1}\leq q^{n-d}(1+\frac{1}{q^{d-i}-1})\leq q^{n-d}(1+\frac{1}{2^{d-i}-1})
\]
and the fact that
\[
\prod_{i=0}^{d-1} (1+\frac{1}{2^{d-i}-1})
\ = \  2\prod_{i=2}^{d} (1+\frac{1}{2^{i}-1})
\ \leq\ 2 e^{\sum_{i=2}^d\frac{1}{2^i-1}}
\ \leq\ 2 e^{\frac{1}{3}\sum_{i=0}^{d-2}\frac{1}{2^i}}
\ \leq\  2 e^{\frac{2}{3}}\leq 4.
\]
\end{proof}

\begin{proposition}\label{supersatnew}
Let $A\in \F_q^{m\times k}$ be a row-balanced matrix of rank $m$. Let $r$ be a positive integer and let $\delta\in (0,1]$. Suppose that for all $n\geq n_0$ and all $S\subseteq \F_q^n$ of cardinality $|S|\geq q^{(1-\delta)n}$ the set $S$ contains a solution to $Ax^\transp=0$ of affine rank $\geq r$. 

Let $\delta'\in (0,\delta)$. Define
\[
n_1=\lceil\max(\tfrac{\delta}{\delta'}n_0, \tfrac{2+\delta}{\delta-\delta'})\rceil,\qquad \epsilon=\delta'\cdot \frac{r-1+2\delta}{\delta},\qquad C=\frac{q^2-2}{4q^2}q^{-\frac{(r-1+\delta)(2+\delta)}{\delta}}.
\]
Then for every $n\geq n_1$ and $S\subseteq \F_q^n$ of size $|S|\geq q^{(1-\delta')n}$, the set $S$ contains at least $Cq^{rn-\epsilon n}$ solutions to $Ax^\transp=0$ of affine rank at least $r$.
\end{proposition}
\begin{proof}
Let $N=\lceil\frac{n\delta'+2}{\delta}\rceil$. Note that $N\leq n$ by definition of $n_1$. Let $S\subseteq \F_q^n$ have cardinality $|S|\geq q^{(1-\delta')n}$. Let us call an $N$-dimensional affine subspace $V$ of $\F_q^n$ \emph{rich} if $|V\cap S|>2q^{(1-\delta)N}$. If $V$ is rich, then $S\cap V$ contains a solution of affine rank at least $r$ since $N\geq n_0$. In fact, since $(S\setminus T)\cap V\geq q^{(1-\delta)N}$ for every set $T$ of size at most $\lceil q^{(1-\delta)N}-1\rceil$, the set $S\cap V$ must contain at least $\lceil q^{(1-\delta)N}\rceil$ solutions of rank at least $r$.

Consider the partition of $\F_q^n$ into the $q^{n-N}$ cosets of a given $N$-dimensional linear subspace. Let $P$ be the proportion of cosets that are rich. Then 
\[
Pq^{n-N}q^N +(1-P)q^{n-N}2q^{(1-\delta)N}\geq |S|\geq q^{(1-\delta')n}.
\]
Rearranging and dividing by $q^n$ gives
\[
P\geq\frac{q^{-\delta'n}-2q^{-\delta N}}{1-2q^{-\delta N}}.
\]
Since $\delta N\geq \delta'n+2$, it follows that 
\[
P\geq \tfrac{q^2-2}{q^2}q^{-\delta'n}.
\]
Let $M$ be the number of pairs $(V,x)$ where $V\subseteq \F_q^n$ is rich and $x=(x_1,\ldots, x_k)$ is a solution of affine rank at least $r$ with $x_1,\ldots, x_k\in V\cap S$. Then 
\[
M\geq P\cdot q^{n-N}\qbin{n}{N}\cdot q^{(1-\delta)N},
\]
since there are $P\cdot q^{n-N}\qbin{n}{N}$ rich subspaces $V$ and for each such $V$ there are at least $q^{(1-\delta)N}$ solutions in $V\cap S$ of rank at least $r$. Conversely, every solution of affine rank at least $r$ is contained in at most $\qbin{n-r+1}{N-r+1}$ rich affine subspaces $V$. We conclude that the number of solutions of affine rank at least $r$ in $S$ is at least
\begin{eqnarray*}
\frac{P\cdot q^{n-N}\cdot q^{(1-\delta)N}\qbin{n}{N}}{\qbin{n-r+1}{N-r+1}}&\geq&\tfrac{1}{4}P\cdot q^{rn-N(r-1+\delta)}\\
&\geq&\frac{q^2-2}{4q^2}q^{n(r-\delta')-N(r-1+\delta)}\\
&\geq&\frac{q^2-2}{4q^2}q^{-(r-1+\delta)\frac{1+\delta}{\delta}}\cdot q^{n(r-\delta')-n\frac{\delta'}{\delta}(r-1+\delta)}\\
&=&Cq^{nr-\epsilon n},\\
\end{eqnarray*}
where we used Lemma~\ref{lem:qbinbound}, the lower bound on $P$, and the fact that $N\leq \frac{\delta'n}{\delta}+\tfrac{2+\delta}{\delta}$ in the three inequalities, respectively.
\end{proof}
\begin{remark}
The obtained value of $\epsilon$ is in general not optimal. In the `cap set case' (i.e. the system is $x_1-2x_2+x_3=0$ and $r=2$), we can take $\delta=1-\log_3(2.7552)$ and obtain $\epsilon=(2+1/\delta)\delta'\approx 14.9\delta'$. However, using the  bound on the arithmetic triangle removal proved in~\cite{fox2017tight}, one can show that we can take $\epsilon=(1+1/\delta)\delta'\approx 13.9\delta'$ as was shown in~\cite{pohoata1905four}.  
\end{remark}

\section{Proof of Theorem A}\label{sec:thmA}
In this section, we will prove Theorem~\ref{thm:mainA}. By Proposition~\ref{prop:extend}, we may restrict ourselves to the case $k=2m+1$:
\begin{theorem}\label{thm:mainbasic}
Let $A\in \F_q^{m\times (2m+1)}$ be a tame, row-balanced matrix. Then the system $Ax^\transp=0$ is temperate. 
\end{theorem}
We will need the following lemma.
\begin{lemma}\label{lem:blurdiagonal}
Let $M$ be an $n\times n$ matrix with $N$ nonzero elements such that every row and every column has at most $k$ nonzero elements. Then $\rank(M)\geq \frac{N}{k^2}$.
\end{lemma}
\begin{proof}
Let $I\subseteq [n]$ be a maximal set of linearly independent columns of $M$. Let $J$ be the set of rows $j$ for which there is a column $i\in I$ with $M_{ij}\neq 0$. Then $|J|\leq k|I|$ and the support of $M$ must be contained in the rows indexed by $J$ by maximality of $I$. On the other hand, the rows in $J$ together contain at most $k|J|\leq k^2|I|$ nonzero entries. So $N\leq k^2|I|$.
\end{proof}

The remainder of this section is devoted to proving Theorem~\ref{thm:mainbasic}. It suffices to show the following for all $r\in [m+1]$:
\begin{equation*}
    \mathbf{P}(r):\ \parbox[t]{0.8\textwidth}{There exists a $\delta_{r}>0$ such that for all sufficiently large $n$, every subset $S\subseteq \F_q^n$ with $|S|\geq q^{(1-\delta_{r})n}$ contains a solution $x$ to $A{x}^\transp=0$ with $\arank(x)\geq r$.}
\end{equation*}
The proof is by induction on $r$. The base case $r=1$ is trivial since for any $v\in S$ a solution of affine rank~$1$ is given by $x=(v,\ldots, v)$. Hence, we can take $\delta_1=1$. 

Now let $r\leq m$ and assume that $\mathbf{P}(r)$ is true. We show that $\mathbf{P}(r+1)$ is true. Let $d=n(q-1)\tfrac{m}{2m+1}$. By Theorem~\ref{thm:slicerank}, there is an $\epsilon\in (0,1)$ such that $m_{q,n,d}=o(q^{(1-\epsilon)n})$ and $m_{q,rn,rd}=o(q^{rn-\epsilon n})$.

By Proposition~\ref{supersatnew}, there is a $\delta_{r+1}\in (0,\delta_r)$ such that every subset $S\subseteq \F_q^n$ with $|S|\geq q^{(1-\delta_{r+1})n}$ contains $\Omega(q^{rn-\epsilon n})$ solutions to $Ax^\transp=0$ of affine rank at least $r$. We take $\delta_{r+1}$ such that moreover $\delta_{r+1}\leq\epsilon$.  

Now suppose, for the sake of contradiction, that for infinitely many $n$ there is a subset $S\subseteq \F_q^n$ of size $|S|\geq q^{(1-\delta_{r+1})n}$ such that $S$ contains no solution to $Ax^\transp=0$ of affine rank at least $r+1$. Consider such pairs $(n,S)$. 

Let $T=\{x\in S^{2m+1}:Ax^\transp=0\}$ be the set of all solutions in $S$. By assumption, every $x\in T$ has affine rank at most $r$. Let $T_1:=\{x\in T:\arank(x)=r\}$. We have $|T_1|=\Omega(q^{rn-\epsilon n})$ by our choice of $\epsilon$. For every $x\in T_1$ there exist, by Proposition~\ref{prop:disjointmaxrank}, disjoint sets $I,J\subseteq [2m+1]$ of size $|I|=|J|=r$ such that 
$\arank \{x_i:i\in I\}=\arank \{x_i:i\in J\}=r$.
Since there are no more than $\binom{2m+1}{r}\binom{2m+1-r}{r}$ possibilities for $(I,J)$, we may assume without loss of generality that 
\[
T_2:=\{x\in T_1:\arank \{x_1,\ldots, x_r\}=\arank\{x_{r+1},\ldots, x_{2r}\}=r\}
\]
has size $|T_2|\geq \binom{2m+1}{r}^{-1}\binom{2m+1-r}{r}^{-1}|T_1|=\Omega(q^{rn-\epsilon n})$.

To emphasise the partition $[2m+1]=\{1,\ldots, r\}\cup\{r+1,\ldots, 2r\}\cup\{2r+1,\ldots 2m+1\}$ of the variables and to simplify notation, we will define $t:=2m+1-2r$ and denote elements from $S^{2m+1}$ by 3-tuples $(x,y,z)$, where $x, y\in S^r$ and $z\in S^t$.

We will construct a random linear map from $\F_q^{S^{2m+1}}$ to $\F_q^{S^r}\times \F_q^{S^r}$ and interpret the elements of $\F_q^{S^r}\times \F_q^{S^r}$ as $|S|^r\times |S|^r$ matrices. We will give an upper bound on the rank of the resulting matrix that is lower than the expected rank for large $n$, thus yielding a contradiction.

Recall that $d=n(q-1)\tfrac{m}{2m+1}$. Let  
\[V=\{g:S\to \F_q:\sum_{z\in S} g(z)z^\alpha=0\quad\forall \alpha\in M_{q,n,d}\}.\]
Since $|S|\geq q^{(1-\delta_{r+1})n}\geq q^{(1-\epsilon)n}$ and $m_{q,n,d}=o(q^{(1-\epsilon)n})$, we have $\dim(V)\geq |S|-m_{q,n,d}>0$ (for $n$ large enough). By Lemma~\ref{lem:support}, there is a subset $S'\subseteq S$ of size $|S'|=\dim(V)$ such that every $h:S'\to \F_q$ extends to a unique $g\in V$. Let $W=V^{\otimes t}\subseteq \{g:S^t\to \F_q\}$ be the linear span of the functions $z\mapsto g_1(z_1)\cdots g_t(z_t)$, where $g_1,\ldots, g_t\in V$. Then for every $g\in W$ 
\begin{equation}
\sum_{z\in S^t} g(z)z_1^{\alpha_1}\cdots z_t^{\alpha_{t}}=0 \text{ if $|\alpha_i|\leq d$ for some $i\in[t]$}
\end{equation}
and every $h:(S')^t\to \F_q$ extends to a unique $g\in W$.

We choose $h:(S')^t\to \F_q$ uniformly at random and let $g\in W$ be the extension of $h$. For any $f:S^{2m+1}\to \F_q$ we define the random $S^r\times S^r$ matrix $M(f)$ as follows:
\[
M(f)_{x,y}:=\sum_{z\in S^t} f(x,y,z)g(z).
\]
For any subset $X\subseteq S^{2m+1}$ we denote by $\one_X:S^{2m+1}\to \F_q$ the $0$--$1$ characteristic function of~$X$. 

\begin{claim}
The matrix $M(\one_T)$ has rank at most $2m_{q,rn,rd}$.
\end{claim}
\begin{proofclaim}
We can write 
\[
\one_T(x,y,z)=\prod_{i=1}^m \prod_{\ell=1}^n \left[1-\left(\sum_{j=1}^r a_{i,j} x_{j,\ell}+\sum_{j=1}^r a_{i,j+r} y_{j,\ell}+\sum_{j=1}^t a_{i,j+2r} z_{j,\ell}\right)^{q-1}\right].
\]
So $\one_T$ is given by a polynomial of degree at most $mn(q-1)$. Expanding the polynomial, we get   
\[
\one_T(x,y,z)=\sum_{(\alpha,\beta,\gamma)\in L} c_{\alpha\beta\gamma} x^\alpha y^\beta z^\gamma
\]
where $L:=\{\ell\in \{0,1,\ldots, q-1\}^{[n]\times [2m+1]}: |\ell|\leq mn(q-1)\}$, and the $c_{\alpha\beta\gamma}\in \F_q$ are certain coefficients. Thus, we have 
\[
M(\one_T)_{x,y}=\sum_{\alpha, \beta} c'_{\alpha,\beta}x^{\alpha}y^{\beta},
\]
where 
\[
c'_{\alpha,\beta}=\sum_{\gamma:(\alpha,\beta,\gamma)\in L} c_{\alpha,\beta,\gamma} \sum_{z\in S} g(z)z^\gamma.
\]
If $|\gamma|\leq td$, then $|\gamma_i|\leq d$ for some $i\in[t]$, and hence $\sum_{z\in S}g(z)z^\gamma=0$ by definition of $g$. It follows that $c'_{\alpha,\beta}=0$ if $|\alpha|+|\beta|\geq mn(q-1)-td=2rd$. So $M(\one_T)$ is given by a polynomial of degree at most $2rd$ in the $rn$ variables $x_{ji}$ and the $rn$ variables $y_{ji}$ (where $(j,i)\in [r]\times [n]$). Hence, it follows from the CLP lemma that $M(\one_T)$ has rank at most $2m_{q,rn,rd}$.
\end{proofclaim}

We will now show the following claim.
\begin{claim}
The matrix $M(\one_T)$ has expected rank $\Omega(q^{rn-\epsilon n})$.
\end{claim}
\begin{proofclaim}
Let $T_3=\{(x,y,z)\in T_2: z\in (S')^t\}$. For any fixed $b\in \F_q^n$ and any $i\in [t]\}$ we have $|\{(x,y,z)\in T_2:z_i=b\}|=O(q^{(r-1)n})$.  Since $|S\setminus S'|\leq m_{q,n,d}=o(q^{(1-\epsilon)n})$, it follows that
\begin{equation}
|T_2\setminus T_3|=o(q^{rn-\epsilon n}),\qquad |T_3|=\Omega(q^{rn-\epsilon n}).
\end{equation}

Let $R=\{(x,y)\in S^{r}\times S^{r}:\exists z\in S^t, (x,y,z)\in T_3\}$. Since for each pair $(x,y)$ there are at most $q^{rt}$ possibilities for $z$ (as each $z_i$ is in the affine span of $\{x_i:i\in[r]\}$), it follows that $|R|\geq \frac{|T_3|}{q^{rt}}=\Omega(q^{rn-\epsilon n})$. Let $(x,y)\in R$. Then  
\[
M(\one_{T_3})_{x,y}=\sum_{z\in S^t} \one_{T_3}(x,y,z)g(z)=\sum_{z\in (S')^t:(x,y,z)\in T_3}h(z).
\]
Since $h$ was chosen uniformly at random and $\{z\in (S')^t: (x,y,z)\in T_3\}$ is nonempty, this sum is nonzero with probability $\tfrac{q-1}{q}$. So the expected size of the support of $M(\one_{T_3})$ is $\tfrac{q-1}{q}|R|$.

For every $x$, there are at most $q^{r(r-1)}$ values of $y$ for which $(x,y)\in R$ since $y_1,\ldots, y_r$ are in the affine hull of $x_1,\ldots, x_r$. Conversely, for every $y$ there are at most $q^{r(r-1)}$ values of $x$ for which $(x,y)\in R$. So by Lemma~\ref{lem:blurdiagonal} we have 
\begin{equation}
\rank M(\one_{T_3})\geq \tfrac{q-1}{q}|R|/(q^{2r(r-1)})=\Omega(q^{rn-\epsilon n}).
\end{equation}

Recall that for any $(x,y,z)\in T\setminus T_2$ we have $\arank(x)\leq r-1$ or $\arank(y)\leq r-1$. Hence, the number of possibilities for $x$ (resp. $y$) is $O(q^{(r-1)n})$ and therefore
\begin{equation}
    \rank M(\one_{T\setminus T_2})=O(q^{(r-1)n}).
\end{equation}

Using the fact that $\rank M(\one_{T_2\setminus T_3})\leq |T_2\setminus T_3|=o(q^{rn-\epsilon n})$, we conclude that 
\[
\rank M(\one_T)\geq \rank M(\one_{T_3})-\rank M(\one_{T_2\setminus T_3})-\rank M(\one_{T\setminus T_2})=\Omega(q^{rn-\epsilon n}).
\]
\end{proofclaim}

Since $m_{q,rn,rd}=o(q^{rn-\epsilon n})$, the two claims yield a contradiction for large enough $n$. This concludes the proof of Theorem~\ref{thm:mainbasic}.

\begin{remark}
In the proof, we took $r=1$ as the base case. However the case $r=2$ also follows directly from Theorem~\ref{thm:slicerank}, so we could have taken $r=2$ as a base case as well. 
\end{remark}

\section{Proof of Theorem B}\label{sec:thmB}
In this section we consider the special case that the finite field is $\F_q$, where $q\in\{2,3\}$.  It turns out that in this case, every balanced system is temperate, by a much simpler argument. The case $q=2$ is implicit in \cite[Lemma 21]{bonin2000size} in the context of extremal matroid theory. 

If $b_0,b_1,\ldots, b_d\in \F_q^n$ and $b_1,\ldots, b_d$ are linearly independent, then $V=\{b_0+\alpha_1b_1+\cdots+\alpha_db_d:\alpha_1,\ldots, \alpha_d\in \F_q\}$ is a \emph{$d$-dimensional affine subspace} of $\F_q^n$. A $1$-dimensional subspace is also called a \emph{line}. 

The main ingredient of the proof of Theorem B is the following lemma.
\begin{lemma}
Fix $q\in\{2,3\}$. Let $d\in \mathbb{Z}_{\geq 1}$. Then there exist $n_d\in \mathbb{Z}_{\geq 1}$, $C_d>0$ and $\delta_d\in (0,1]$ such that for all $n\geq n_d$, every $S\subseteq \F_q^n$ of size $|S|\geq C_dq^{(1-\delta_d)n}$ contains an affine $d$-dimensional subspace.
\end{lemma}
The case $q=3$, the \emph{multidimensional cap set problem}, was proved in~\cite{fox2019popular} by using the arithmetic triangle removal lemma from~\cite{fox2017tight}. A more direct proof (but with slightly worse bounds) due to Jop Bri\"et \cite{Jop} (personal communication) is given below. 

\begin{proof}
If $q=3$, we can take $n_1=1$, $C_1=3$, and $\delta_1=1-\log_3(2.756)$ by the cap set theorem. If $q=2$, we can take $n_1=1$, $C_1=2$, and $\delta_1=1$. 

Now let $t\geq 2$ and suppose that the statement is true for $d=1,\ldots, t-1$. We define 
\[C_t:=q(C_1+C_{t-1}),\quad \delta_t:=\frac{\delta_1\delta_{t-1}}{2+\delta_1+\delta_{t-1}}, \quad n_t:=\left\lceil\max\left(\frac{2+\delta_1+\delta_{t-1}}{\delta_{t-1}}n_{1},\frac{2+\delta_1+\delta_{t-1}}{2+\delta_1}n_{t-1}\right)\right\rceil.\]
Note that $0<\delta_t\leq1$.
Let $n\geq n_t$ and let $S\subseteq \F_q^n$ have size $|S|\geq C_tq^{(1-\delta_t)n}$. We define
\[
m=\left\lfloor\frac{\delta_{t-1}}{2+\delta_1+\delta_{t-1}}n\right\rfloor,\qquad k=\left\lceil\frac{2+\delta_1}{2+\delta_1+\delta_{t-1}}n\right\rceil=n-m.
\]
By the choice of $n_t$ we have $m\geq n_{1}$ and $k\geq n_{t-1}$. For any $x\in \F_q^k$ we denote $S_x:=\{y\in \F_q^m: (x,y)\in S\}$. Let $T=\{x\in \F_q^k:|S_x|\geq C_1q^{(1-\delta_1)m}\}$. We have
\[
|S|=\sum_{x\in T} |S_x|+\sum_{x\not\in T} |S_x|\leq |T|q^m+C_1q^{k+(1-\delta_1)m}.
\]
Since $\delta_1 m+1\geq \delta_tn$, we have $(1-\delta_t)n+1\geq k+(1-\delta_1)m$. Similarly, we have $(1-\delta_t)n+1\geq 3m+(1-\delta_{t-1})k$. It follows that
\[
|S|\geq C_tq^{(1-\delta_t)n}\geq C_{1}q^{k+(1-\delta_1)m}+C_{t-1}q^{3m+(1-\delta_{t-1})k}.
\]
Hence, it follows that 
\[
|T|\geq \frac{|S|-C_1q^{k+(1-\delta_1)m}}{q^m}\geq q^{2m}C_{t-1}q^{(1-\delta_{t-1})k}.
\]
For every $x\in T$, there is a line $\ell$ contained in $S_x$. Since the number of lines in $\F_q^m$ is less than $q^{2m}$, there is a subset $T'\subseteq T$ of size $|T'|\geq |T|q^{-2m}$ and a line $\ell\subseteq \F_q^m$ such that $\ell\subseteq S_x$ for all $x\in T'$. Since $|T'|\geq C_{t-1}q^{(1-\delta_{t-1})k}$ there is a ($t-1$)-dimensional subspace $V\subseteq T'$. Hence $\{(x,y):x\in V, y\in \ell\}$ is a $t$-dimensional subspace contained in $S$.
\end{proof}

We restate Theorem~B.
\begin{theorem}
Let $q\in\{2,3\}$ and let $A\in \F_q^{m\times k}$ be a row-balanced matrix of rank $m$. Then $Ax^\transp=0$ is temperate.
\end{theorem}
\begin{proof}
Let $d=k-m-1$. Then $\F_q^d$ contains a generic solution $y$ to $Ax^\transp=0$ by Lemma~\ref{lem:lowdim}. So $y$ has affine rank $m-k$. 

Now let $n\geq n_d$ and let $S\subseteq \F_q^n$ have size $|S|\geq C_dq^{(1-\delta_d)n}$. Then $S$ contains a $d$-dimensional affine subspace $V=\{b_0+\alpha_1b_1+\cdots+\alpha_db_d:\alpha_1,\ldots, \alpha_d\in \F_q\}$. Setting $x_i=b_0+y_{i1}b_1+\cdots+y_{id}b_d$, we obtain a solution to $Ax^\transp=0$ where $x$ has affine rank $k-m$ and $x_1,\ldots, x_k\in S$.
\end{proof}

\section{Affine copies}\label{sec:aff}
Given a balanced matrix $A\in \F_q^{m\times k}$ of rank $m$, there is a generic solution $z=(z_1,\ldots, z_k)$ to $Ax^\transp=0$ in $\F_q^{k-m-1}$ by Lemma~\ref{lem:lowdim}. Conversely, if $z=(z_1,\ldots, z_k)\in (\F_q^r)^k$, then $z$ is a generic solution to $Ax^\transp=0$, where $A\in \F_q^{m\times k}$ has rows that are a basis of $\Annbal(z)$. Instead of saying that $S\subseteq \F_q^n$ has a generic solution to $Ax^\transp=0$, we can equivalently say that $S$ contains an affine copy of $z$. Thus, we can reformulate Theorem~\ref{thm:mainA} and Theorem~\ref{thm:mainB} and obtain Corollary~C. 

\begin{definition}
Let $z_1,\ldots, z_k\in \F_q^r$ and $x_1,\ldots, x_k\in \F_q^n$. We say that $x=(x_1,\ldots, x_k)$ is an \emph{affine copy} of $z=(z_1,\ldots, z_k)$ if there is an injective linear map $L:\F_q^r\to \F_q^n$ and a vector $a\in \F_q^n$ such that $x_i=L(z_i)+a$ for all $i\in [k]$.
\end{definition}
Note that if $x$ is an affine copy of $z$, then any translate of $x$ is an affine copy of any translate of $z$.

\begin{lemma}\label{lem:affinecopy}
Let $z=(z_1,\ldots, z_k)\in (\F_q^r)^k$. Let $A\in \F_q^{m\times k}$ be such that the rows form a basis of $\Annbal(z)$. Let $n\geq r$ and let $x=(x_1,\ldots, x_k)\in (\F_q^n)^k$. Then $x$ is an affine copy of $z$ if and only if $x$ is a generic solution to $Ax^\transp=0$. 
\end{lemma}
\begin{proof}By translating, we may assume without loss of generality that $z_k=0$ and $x_k=0$. Note that $z=(z_1,\ldots, z_k)$ is a generic solution to $Ax^\transp=0$. 

Suppose that $x$ is an affine copy of $z$. Then $x_i=L(z_i)$ for all $i\in [k]$, where $L:\F_q^r\to \F_q^n$ is an injective linear map. It is easy to verify that $\Annbal(x)=\Annbal(z)$. So $x$ is a generic solution to $Ax^\transp=0$. 

Conversely, suppose that $x$ is a generic solution to $Ax^\transp=0$. Define $L:\spn\{z_1,\ldots, z_{k-1}\}\to \spn\{x_1,\ldots, x_{k-1}\}$ by
\[
L(\sum_{i\in [k-1]} \lambda_i z_i):=\sum_{i\in [k-1]} \lambda_i x_i.
\]
This is a well-defined injective linear map, since
\begin{align*}
    \sum_{i\in [k-1]} \lambda_i z_i=0&\iff
    (\lambda_1,\ldots, \lambda_{k-1}, -\sum_{i\in [k-1]}\lambda_i)\in \Annbal(z)\\
    &\iff (\lambda_1,\ldots, \lambda_{k-1}, -\sum_{i\in [k-1]}\lambda_i)\in \Annbal(x)\\
    &\iff \sum_{i\in [k-1]} \lambda_i x_i=0.
\end{align*}
Since $n\geq r$, we can extend $L$ to an injective linear map $\F_q^r\to \F_q^n$. So $x$ is an affine copy of $z$.
\end{proof}

We can now prove Corollary~\ref{cor:mainC} as a direct consequence of Theorem~\ref{thm:mainA} and Theorem~\ref{thm:mainB}.
\begin{proof}[Proof of Corollary~\ref{cor:mainC}]
Let $A\in \F_q^{m\times k}$ be a matrix such that the rows form a basis of $\Annbal(z_1,\ldots, z_k)$. So $z=(z_1,\ldots, z_k)$ is a generic solution to $Ax^\transp=0$. If $q\not\in \{2,3\}$, then the premise of Corollary~\ref{cor:mainC} implies by Lemma~\ref{lem:bases} that for every $i\in [k]$ the matrix $A$ has two disjoint bases in $[k]\setminus\{i\}$. Hence, by Lemma~\ref{tamechar}, the matrix $A$ is tame.

By Lemma~\ref{lem:affinecopy}, it follows directly from Theorem~\ref{thm:mainA} and Theorem~\ref{thm:mainB} that there is a $\delta>0$ such that any subset $S\subseteq \F_q^n$ without affine copy of $z$ has size $|S|=O(q^{(1-\delta)n})$.
\end{proof}

\section{Concluding remarks}\label{sec:conclusion}
An advantage of considering generic solutions over \mytag{}-shapes, is that they can be used to obtain not just one, but many solutions given a sufficiently large set $S\subseteq \F_q^n$. In particular, we obtain Corollary~\ref{cor:mainD} which we restate here.

\begin{corollary}
    Suppose that the coefficient matrix $A$ of \mytag{} is tame. Then there exist $\delta>0$ and $C>0$ such that for all positive $\delta'<\delta$ the following holds: if $S\subseteq \F_q^n$ has size $|S|\geq q^{(1-\delta')n}$, then $S$ has $\Omega(q^{(k-m)n-C\delta'n})$ solutions to \mytag{}. 
\end{corollary}
\begin{proof}
This follows directly from Theorem~\ref{sec:thmA} and Proposition~\ref{supersatnew}.
\end{proof}

We conclude with an open problem for which we think both a positive and a negative answer would be very interesting.
\begin{open problem}
Is it true that \mytag{} is temperate for all balanced matrices $A$?
\end{open problem}
A positive answer would be a very strong result, since we do not even know if the system corresponding to $4$-term progressions is temperate. As a final note, we remark that the situation where $A$ is tame, corresponds to a strict subset of systems of linear forms of Cauchy-Schwarz complexity~$1$ (see: \cite{GowersWolf2007}). Perhaps this class would be a good candidate for proving a generalisation. 

%%% AUTHOR:
%%% Bibliography goes here. Note that the arXiv cannot process bibtex
%%% or biber bibliographies.  Example of acceptable bibliograpy format:
\bibliographystyle{amsplain}
\providecommand{\noopsort}[1]{}

%%% AUTHOR: Include a short description of each author following the
%%% structure below. Use the same short tags used previously.  
%%% Use \imageat{} and \imagedot{} instead of "@" and "." in
%%% email addresses-this replaces the symbols with graphics to avoid 
%%% e-mail address harvesting from the .pdf file
\begin{dajauthors}
\begin{authorinfo}[dion]
  Dion Gijswijt\\
  Delft Institute of Applied Mathematics\\
  Delft, The Netherlands\\
  d\imagedot{}c\imagedot{}gijswijt\imageat{}tudelft\imagedot{}nl \\
  \url{https://diamhomes.ewi.tudelft.nl/~dgijswijt/}
\end{authorinfo}
\end{dajauthors}

\end{document}